\documentclass[letter]{amsart}
\usepackage{hyperref}
\usepackage{amsmath,amsfonts,amsthm,amssymb}
\usepackage{graphicx}
\usepackage{color}
\usepackage{caption}
\usepackage{subcaption}

\theoremstyle{plain}
\newtheorem{theorem}{Theorem}
\newtheorem{hypothesis}{Hypothesis}
\newtheorem{lemma}{Lemma}

\theoremstyle{remark}
\newtheorem*{remark}{Remark}
\newtheorem*{remarks}{Remarks}

\theoremstyle{definition}
\newtheorem*{definition}{Definition}

\newtheorem*{conjecture}{Conjecture}

\newtheoremstyle{hyp}
{}  % abovespace
{}  % belowspace
{\itshape} % bodyfont
{0pt}    % indent (empty value is the same as 0pt)
{\bfseries} % headfont
{)}     % headpunct
{5pt plus 1pt minus 1pt} % headspace
{}     % custom-head-spec

\theoremstyle{hyp}
\newtheorem{h}{(H\kern0pt}

\DeclareMathOperator{\supp}{supp}

\DeclareMathOperator{\id}{Id}

\DeclareMathOperator{\const}{constant}

\DeclareMathOperator{\interior}{Int}
\DeclareMathOperator{\dv}{div}
\DeclareMathOperator{\dist}{dist}
\DeclareMathOperator{\e}{e}

\newcommand{\M}{\mathbb{M}}
\newcommand{\K}{\mathbb{K}}
\newcommand{\Z}{\mathbb{Z}}
\newcommand{\R}{\mathbb{R}}
\newcommand{\dd}{\,\mathrm{d}}

\newcommand{\corner}{\,\raisebox{.2ex}{\scalebox{1.15}{$\llcorner$}}\,}
\newcommand{\symmdiff}{\,\raisebox{.2ex}{\scalebox{.7}{$\triangle$}}\,}

\let\oldmarginpar\marginpar
\renewcommand\marginpar[1]{\-\oldmarginpar[\raggedleft\scriptsize #1]%
{\raggedright\scriptsize #1}}

\begin{document}

\title[On the structure of phase transition maps]{On the structure of phase transition maps for three or more coexisting phases}
\author{Nicholas D.\ Alikakos}
\address{Department of Mathematics\\ University of Athens\\ Panepistemiopolis\\ 15784 Athens\\ Greece \and Institute for Applied and Computational Mathematics\\ Foundation of Research and Technology -- Hellas\\ 71110 Heraklion\\ Crete\\ Greece}
\email{\href{mailto:nalikako@math.uoa.gr}{\texttt{nalikako@math.uoa.gr}}}
\thanks{The author was partially supported through the project PDEGE -- Partial Differential Equations Motivated by Geometric Evolution, co-financed by the European Union -- European Social Fund (ESF) and national resources, in the framework of the program Aristeia of the `Operational Program Education and Lifelong Learning' of the National Strategic Reference Framework (NSRF)}
\date{}
\maketitle

\section{Introduction}
\label{section1}
This paper is partly based on a lecture delivered by the author at the ERC workshop ``Geometric Partial Differential Equations'' held in Pisa in September 2012. What is presented in the following is an expanded version of that lecture.

Specifically, we consider the system
\begin{equation}\label{system}
\Delta u - W_u(u) = 0, \text{ for } u : \R^n \to \R^m,
\end{equation}
with $W \in C^2(\R^m; \R_+)$, $W \geq 0$, where $W_u := (\partial W / \partial u_1, \dots, \partial W / \partial u_n)^{\top}$, and occasionally with additional hypotheses introduced later on. We refer the reader to Part I in \cite{alikakos-fusco-pisa} for general information and motivation for system \eqref{system}.

The paper is organized as follows. In Section \ref{section2} we present the basics of the problem for general potentials, in Section \ref{section3} we study symmetric potentials for the phase transition model and establish the existence of equivariant connection maps, in Sections \ref{section4} and \ref{section5} we present and prove a related Bernstein-type theorem, and in Section \ref{section6} we discuss the hierarchical structure of the equivariant connection maps. Sections \ref{section2}, \ref{section3}, and \ref{section6} are restricted to statements of results with explanations but without proofs, referring to published papers or preprints for the details. In contrast, in Sections \ref{section4} and \ref{section5} we give the background and detailed proofs.

The author would like to acknowledge the warm hospitality of the Department of Mathematics of Stanford University in the spring semester of 2012, during which part of this paper was written. Special thanks are due to Rafe Mazzeo, George Papa\-nicolaou, Lenya Ryzhik, Rick Schoen, and Brian White. Section \ref{section4} is very much influenced from discussions with Rick Schoen and lectures and material provided by Brian White.

\section{The basics for general potentials}
\label{section2}
We recall some known facts for \eqref{system}. The system is the Euler--Lagrange equation for the \emph{free energy functional}
\[ J(u;\R^n) := \int_{\R^n} \left( \frac{1}{2} |\nabla u|^2 + W(u) \right) \dd x, \]
where $\nabla u = (\partial u_i / \partial u_j )$, for $i = 1, \dots, m$, $j = 1, \dots, n$, and $|\cdot|$ is the Euclidean norm of the matrix.

One important difference of \eqref{system} with its scalar counterpart (for $m=1$) is that in that case the structure of bounded entire solutions does not depend very much on $W$. In contrast, for the system there two distinguished examples with very distinct behavior: The class of \emph{phase transition potentials}, that is, $W$'s with a finite number of global minima (wells) $a_1, \dots, a_N$, with $W(a_i) = 0$, and the class of \emph{Ginzburg--Landau potentials}, for example, the potential $W(u) = \frac{1}{4} (|u|^2 - 1)^2$ (disconnected versus connected zero sets). In the phase transition case and under suitable rescaling the free energy concentrates on minimal hypersurfaces or Plateau complexes (see \cite{baldo}), while in the Ginzburg--Landau case it concentrates on higher-codimension objects known as vortices, and otherwise the solution converges to a harmonic map (see \cite{bethuel-brezis-helein}).

Equation \eqref{system} can be written as a divergence-free condition, that is,
\[ \dv T = (\nabla u)^\top (\Delta u - W_u(u)) = 0, \]
for the stress-energy tensor
\[
T_{ij} (u, \nabla u) := u_{,i} \cdot u_{,j} - \delta_{ij} \left( \frac{1}{2} |\nabla u|^2 + W(u) \right).\footnote{The sharp-interface limit $T_\varepsilon \to T_0 = \sigma (\nabla d \otimes \nabla d - \id)$ is the orthogonal projection to the tangent space of the interface $S$ separating the two phases, where $d = \dist (x,S)$ and $\sigma$ is the associated interface energy.}
\]
In this context it was introduced in \cite{alikakos-basic-facts}, but as it turns out it is a particularization of a general formalism well-known to the physicists \cite{landau-lifshitz}.

The divergence-free formulation has certain important consequences. For example, one can derive the monotonicity formula
\[
\frac{\dd}{\dd R} \left( \frac{1}{R^{n-2}} \int_{|x-x_0|<R} \left( \frac{1}{2} |\nabla u|^2 + W(u) \right) \dd x \right) \geq 0,
\]
from which Liouville-type theorems follow (see \cite{alikakos-basic-facts}). For instance,
\[ \begin{cases}
J(u; B_R) = o(R^{n-2}) \text{ as } R \to +\infty, \text{ for } n \geq 3, &\text{implies} \quad u \equiv \const, \smallskip\\
J(u; B_R) = o(\log R) \text{ as } R \to +\infty, \text{ for } n = 2, &\text{implies} \quad u \equiv \const. 
\end{cases} \]
In particular,
\begin{equation}\label{liouville}
\int_{\R^n} \left( \frac{1}{2} |\nabla u|^2 + W(u) \right) \dd x < +\infty \quad \text{implies} \quad u \equiv \const,
\end{equation}
for $n \geq 2$. Note that \eqref{liouville} is a source of difficulty for constructing solutions to \eqref{system} via the direct method. It was Farina \cite{farina} who first derived the monotonicity formula above and its implication \eqref{liouville} in the context of the Ginzburg--Landau system and Modica \cite{modica1} who derived \eqref{liouville} for $m=1$.

In the scalar ODE case ($n=1$, $m=1$), for solutions of \eqref{system} with limits at infinity, one has the elementary \emph{equipartition relation}
\[
\frac{1}{2} |u_x|^2 = W(u).
\]
In the scalar PDE case ($n \geq 2$, $m=1$), Modica \cite{modica1} established the estimate
\begin{equation}\label{modica-estimate}
\frac{1}{2} |\nabla u|^2 \leq W(u),
\end{equation}
(see also \cite{caffarelli-garofalo-segala}). The analog of estimate \eqref{modica-estimate} is false for systems in general. All the known counterexamples (see \cite[pp.~389--390]{farina}) involve Ginzburg--Landau potentials. One implication of \eqref{modica-estimate} would be the stronger monotonicity formula
\[
\frac{\dd}{\dd R} \left( \frac{1}{R^{n-1}} \int_{|x-x_0|<R} \left( \frac{1}{2} |\nabla u|^2 + W(u) \right) \dd x \right) \geq 0,
\]
already known for the scalar case (see \cite{modica2}).

Another implication of the divergence-free formulation is a Pohozaev-type identity (see \cite{alikakos-faliagas1})
\[
\frac{n-2}{2} \int_\Omega |\nabla u|^2 \dd x + n\int_\Omega W(u) \dd x + \frac{1}{2} \int_{\partial \Omega} (x - x_0) \cdot \nu\, |\nabla u|^2 \dd S = 0,
\]
where $\nu$ is the outward normal and $x_0 \in \Omega$ arbitrary, for solutions of the system
\[
\begin{cases}
\Delta u - W_u(u) = 0, &\text{in } \Omega \subset \R^n,\smallskip\\
u=a, &\text{on } \partial \Omega, \text{ with } W(a) = 0.
\end{cases}
\]

Gui \cite{gui} has developed certain identities for the system, which he calls `Hamiltonian', and points out their relationship with the classical Pohozaev identity (see, for example, \cite{evans}). His identities can also be derived via the stress-energy tensor. Here is a sample: Let $n=2$, $m=2$, and let $u$ be a solution to \eqref{system} satisfying the estimate
\[
|u(x_1, x_2) - a_\pm| \leq C \e^{-c |x_1|}, \text{ for all } x_2 \in [M,N],
\]
with $-\infty \leq M,N \leq +\infty$ and $W(a_\pm) = 0$. Then,
\[
\int_\R \left( \frac{1}{2} \left( |u_{x_1}(x_1,x_2)|^2 - |u_{x_2}(x_1,x_2)|^2 \right) - W(u(x_1,x_2)) \right) \dd x_1 = \const,
\]
for all $x_2 \in (M,N)$.

Having asymptotic information on the solution along certain hyperplanes as $|x| \to +\infty$ and also convergence to a minimum of $W$ away from them (see Section \ref{section4}), one can measure the flux of the stress-energy tensor over large spheres in order to derive balance conditions relating the angles between the hyperplanes, thus deriving rigidity-type results. For the phase transition case and for a triple-well potential, Gui \cite{gui} has derived such a result in the planar case $n=2$, $m=2$, thus relating the angles of a triple junction to the surface energies. This was extended to the three-dimensional case $n=3$, $m=3$, in \cite{alikakos-antonopoulos-damialis}. Related also is the work of Kowalczyk, Liu, and Pacard \cite{kowalczyk-liu-pacard}.

\section{Symmetric phase transition potentials -- Existence of equivariant connection maps}
\label{section3}
In this section we restrict ourselves to the phase transition case for potentials that respect the symmetries of a finite reflection group $G$ acting on $\R^n$ (see \cite{grove-benson}) and we look for equivariant solutions
\[
u(gx) = g u(x), \text{ for all } x \in \R^n \text{ and } g \in G.
\]

The first results in this direction are due to Bronsard, Gui, and Schatzman \cite{bronsard-gui-schatzman} for $n=2$, $m=2$, and $G$ the group of reflections of the equilateral triangle. Later, the work was extended by Gui and Schatzman \cite{gui-schatzman} to $n=3$, $m=3$, and $G$ the group of symmetries of the regular tetrahedron. These two special groups are particularly important as they are related to triple junctions on the plane and to quadruple junctions in three-dimensional space, which are minimal objects (cones) for the related sharp-interface problem. 

In work with Fusco \cite{alikakos-fusco-arma} we considered the general case of a reflection group and looked for an abstract result. Consider the following very general hypotheses.

\begin{hypothesis}[$N$ nondegenerate global minima]\label{hypothesis1}
The potential $W$ is of class $C^2$ and satisfies $W(a_i)=0$, for $i=1,\ldots,N$, and $W>0$ on $\R^m \setminus \{a_1,\dots a_N\}$. Furthermore, there holds $v^\top\partial^2 W(u)\, v \geq 2 c^2 |v|^2$, for $v\in\R^m$ and $i=1,\ldots, N$.
\end{hypothesis}

\begin{hypothesis}[Symmetry]\label{hypothesis2}
The potential $W$ is invariant under a finite reflection group $G$ acting on $\R^m$ (Coxeter group), that is,
\[
W(gu) = W(u), \text{ for all } g \in G \text{ and } u \in \R^m.
\]
Moreover, we assume that there exists $M>0$ such that $W(su) \geq W(u)$, for $s\geq 1$ and $|u|=M.$
\end{hypothesis}

\begin{hypothesis}[Location and number of global minima]\label{hypothesis3}
Let $F \subset \R^m$ be a fundamental region of $G$. We assume that the closure $\overline{F}$ contains a single global minimum of $W,$ say $a_1$, and let $G_{a_1}$ be the subgroup of $G$ that leaves $a_1$ fixed.
\end{hypothesis}
We set
\[
D = \interior \{ \cup g \overline{F} \mid g \in G_{a_1} \},
\]
and notice that by the invariance of $W$ it follows that the number of minima of $W$ is
\[
N= \frac{|G|}{|G_{a_1}|},
\]
where here $| \cdot |$ is the order of the group.

We recall from \cite{alikakos-fusco-arma} several examples. For $G=\mathcal{H}^{3}_{2}$, the group of symmetries of the equilateral triangle on the plane, we can take as $F$ the $\frac{\pi}{3}$ sector. If $a_1 \in F$, then $N=6$, while if $a_1$ is on the walls, then $N=3$. In higher dimensions we have more options since we can place $a_1$ in the interior of $\overline{F}$, in the interior of a face, on an edge, and so on. For example, if $G=\mathcal{W}^*$, the group of symmetries of the cube in three-dimensional space, then $|G|=48$. If the cube is situated with its center at the origin and its vertices at the eight points $(\pm 1, \pm 1, \pm 1)$, then we can take as $F$ the simplex generated by $s_1 = e_1 + e_2 + e_3$, $s_2 = e_2 + e_3$, and $s_3 = e_3$, where the $e_i$'s are the standard basis vectors. We have then the following options:
\begin{enumerate}
\item At the origin, $N=1$.
\item On the edge $s_3$, $N=6$.
\item On the edge $s_1$, $N=8$.
\item On the edge $s_2$, $N=12$.
\item In the interior of a face, $N=24$.
\item In the interior of the fundamental region, $N=48$.
\end{enumerate}

We have the following theorem.
\begin{theorem}[\cite{alikakos-fusco-arma, alikakos-cpde, fusco}]\label{theorem1}
Under Hypotheses \ref{hypothesis1}--\ref{hypothesis3}, there exists a classical entire equivariant solution $u: \R^n \to \R^m$ to system \eqref{system} such that
\begin{enumerate}
\item $|u(x)-a_1| \leq K \mathrm{e}^{-k \dist (x,\partial D)}$, for $x \in D$ and for positive constants $k$, $K,$ \medskip
\item $u(\overline{F}) \subset \overline{F}$ and $u(D) \subset D$ (positivity).\footnote{Smyrnelis has established that $u(F) \subset F$ for certain groups (personal communication).}
\end{enumerate}
As a consequence of {\rm (i)}, the solution $u$ \emph{connects} the $N=|G|/|G_{a_1}|$ global minima of $W$ in the sense that
\[ \lim_{\lambda \to +\infty} u(\lambda g \eta) = g a_1, \text{ for all } g \in G,\]
uniformly for $\eta$ in compact subsets of $D\cap\mathbb{S}^{n-1}.$
\end{theorem}

\begin{remark}
We need a clarification concerning the dimensions $n,m$ and the group $G$ that is acting by Hypothesis \ref{hypothesis2} on the target. If $n \geq m$, then the group $G$ can be embedded in the domain space via a natural homomorphism. For example, consider $n=3$, $m=2$, and $G$ the group of symmetries of the equilateral triangle. On the other hand, if $n<m$, the existence of such a homomorphism is more problematic and in general there is no such embedding. For example, consider $n=2$, $m=3$, and take as $G$ the group associated to the tetrahedron. For relevant information we refer to Bates, Fusco, and Smyrnelis \cite{bates-fusco-smyrnelis}. Our notation $u(\overline{F})$ and $u(\lambda g \eta)$ tacitly assumes the homomorphism in the case $n \neq m$.
\end{remark}

The above theorem was proved in Alikakos and Fusco \cite{alikakos-fusco-arma} under an additional hypothesis. Subsequently, the author gave a simplified proof in \cite{alikakos-cpde} and, finally, in Fusco \cite{fusco} the extra hypothesis was removed and the theorem was proved under the hypotheses above.

As it was mentioned in \eqref{liouville}, there holds
\[
J(u; \R^n) = +\infty, \text{ where } J(u, \Omega) = \int_{\Omega} \left( \frac{1}{2} |\nabla u|^2 + W(u) \right) \dd x,
\]
for the solution constructed above. However, the solutions are constructed variationally and possess the following minimization property (see \cite{alikakos-fusco-hierarchy}), which defines the notion of a \emph{local minimizer} (cf.\ \cite{alberti-ambrosio-cabre}), that is,
\begin{equation}\label{global-minimizer}
J(u; \Omega) = \min J(v; \Omega), \text{ such that } v=u \text{ on } \partial \Omega,
\end{equation}
over all domains $\Omega$ that are bounded, smooth, and open, but not necessarily symmetric, and over all equivariant positive maps $v$ (cf.\ (ii) in the statement of Theorem \ref{theorem1}) in $W^{1,2} (\Omega; \R^m)$.

\begin{remarks}
The constructed solution $u$ is a minimizer in the equivariant positive class. For the equilateral triangle group on the plane and the regular tetrahedron group in three-dimensional space, one would expect that the solution constructed is a minimizer in the class of all $W^{1,2} (\Omega; \R^n)$-maps.

It appears that the positivity of $u$ is not an implication of the minimizing property \eqref{global-minimizer} for arbitrary $G$, without extra hypotheses on $W$.

Looking for equivariant solutions is of course a convenience. However, at this point all the existence results for system \eqref{system} known to the author involve hypotheses of symmetry (except for $n=1$).
\end{remarks}

\section{A related Bernstein-type theorem -- Background}
\label{section4}
We recall De Giorgi's conjecture \cite{de-giorgi} for the scalar equation, for $n \geq 2$, $m=1$.
\begin{conjecture}[De Giorgi]
\label{de-giorgi-conjecture}
For the equation 
\[ \Delta u - W'(u)=0,\] 
with $W(u)=\frac{1}{4}(u^2-1)^2$, under the hypotheses that $u: \R^n \to \R$ is in $C^2(\R^n;[-1,1])$ with ${\partial u} / {\partial x_n} > 0$, is it true that the level sets of $u$ are hyperplanes, at least for $n \leq 8$?
\end{conjecture}

A more restricted version of the conjecture involves the additional hypothesis $\lim_{x_n \to \pm \infty} u(x) = \pm 1$.

This conjecture was established by Ghoussoub and Gui \cite{ghoussoub-gui} for $n=2$, Ambrosio and Cabr\'e \cite{ambrosio-cabre} for $n=3$, and Savin \cite{savin} in the restricted form for $4 \leq n \leq 8$. Finally, it was disproved for $n \geq 9$ by del Pino, Kowalczyk, and Wei \cite{delpino-kowalczyk-wei,delpino-kowalczyk-wei-pnas}. We refer to the survey paper by Farina and Valdinoci \cite{farina-valdinoci}, where in addition several extensions to a variety of related equations are given. 

Some of the ingredients behind the formulation of this conjecture are
\begin{enumerate}
\item the Bernstein theorem for graphs,
\item the relationship between monotonicity and stability,
\item the solution of the ODE (the \emph{heteroclinic connection})
\[ \frac{\dd^2 U}{\dd \eta^2} - W'(U) = 0, \text{ with } \lim_{\eta \to \pm \infty} U(\eta)=\pm 1\]
(unique up to translations),
\item the phase transition problem for two phases.
\end{enumerate}

The conclusion in the conjecture is equivalent to showing that 
\[ u(x) = U \left( \frac{a \cdot x - c}{\sqrt{2}} \right),\] 
for some $a \in \R^n$, with $|a|=1$, and $c \in \R$, that is, $u(x) = U(Px)$, where $P$ is the orthogonal projection to the normal direction of the level sets.

In formulating the analog of the conjecture for systems one should keep in mind that
\begin{enumerate}
\item Tangent planes are special cases of tangent cones. Moreover, minimizing tangent cones have cylindrical structure, that is, 
\[ C = V \times \tilde{C}, \text{ with } \tilde{C} \text{ minimizing in } V^{\bot},\]
and the cone $C$ is translation invariant `along $V$' ($V=\{ 0 \}$ is an option).
Also, the Liouville object in this context is the cone at infinity. The Bernstein-type theorem therefore should involve a cone.
\item Monotonicity is not related in general to stability for systems.
\item The solution $u: \R^{n} \to \R^{n-k}$ \emph{a posteriori} should be of the form
\[ u(x) = \hat{u} (Px), \]
where $P$ is an orthonormal projection on an $(n-k)$-dimensional plane, $\hat{u} : \R^{n-k} \to \R^{n-k}$ is a \emph{connection} map, equivariant as in Theorem \ref{theorem1} in Section \ref{section3}.
\item For three or more phases the order parameter should be a vector since otherwise there is no connection between the extreme phases. For example, for coexistence of three phases we need at least a two-dimensional order parameter, thus a partitioning of $\R^2$ in three parts.
\item For the analog of the restricted conjecture we refer to Section \ref{section5} in the present paper.
\end{enumerate}

Our purpose in this section is to present a sample of such Bernstein-type theorems for the simplest nontrivial case, the triple junction in $\R^3$, that corresponds to one of the two singular minimizing cones in $\R^3$ (see \cite{taylor}). The formulation of such theorems in terms of cones goes back to Fleming \cite{fleming} and is subsequently developed in Morgan \cite{morgan}. However, here we also want to emphasize partitions as the natural setup. For this reason we present in detail White's approach \cite{white} because it improves Almgren's \cite{almgren} and because of its simplicity, and also for making our treatment as self-contained as possible. Our presentation here is also based on Chan's thesis \cite{chan}, written under White's supervision.

Finally, we mention the paper of Fazly and Ghoussoub \cite{fazly-ghoussoub}, which extends the methods of the scalar equation to systems, as far as this can possibly be done, by assuming certain monotonicities on the components of the solution which amount to 
\[ \frac{\partial^2 W(u)}{\partial u_i \partial u_j} \leq 0 , \text{ for } i \neq j. \]
For a special system of two equations, we mention the papers of Berestycki, Lin, Wei, and Zhao \cite{berestycki-lin-wei-zhao}, Berestycki, Terracini, Wang, and Wei \cite{berestycki-terracini-wang-wei}, Farina \cite{farina2}, and Farina and Soave \cite{farina-soave}.

We recall here some basic background on partitions and geometric measure theory (see \cite{white}, \cite{simon1}, \cite{white2}).

\subsection{Minimizing partitions}
Consider an open set $U \subset \R^n$ occupied by $N$ immiscible fluids, or phases. Associated to each pair of phases $i$ and $j$ there is a surface energy density $e_{ij}$, with $e_{ij}>0$ for $i \neq j$, and $e_{ij}=e_{ji}$, with $e_{ii}=0$. Hence, if $A_i$ denotes the subset of $U$ occupied by phase $i$, then $U$ is the disjoint union
\[ U = A_1 \cup A_2 \cup A_2 \cup \cdots \cup A_N, \]
and the energy of the partition $A=\{ A_i \}^N_{i=1}$ is
\[ E(A)=\sum_{0<i<j \leq N} e_{ij}\, \M (\partial A_i \cap \partial A_j), \]
where $\M$ (for \emph{mass}) stands for the measure of the interface. For $n=3$ it will simply be the area of $\partial A_i \cap \partial A_j$.

If $U$ is unbounded, for example $U=\R^n$ (we say then that $A$ is \emph{complete}), the quantity above in general will be infinity. Thus, for each $W$ open, with $W \subset \subset U$, we consider the energy
\[
E(A;W)=\sum_{0<i<j \leq N} e_{ij}\, \M (I_{ij} \cap W), \text{ where } I_{ij}:=\partial A_i \cap \partial A_j.
\]

\begin{definition}
The partition $A$ is a \emph{minimizing} $N$-partition if given any $W \subset \subset U$ and any $N$-partition $A'$ of $U$ with
\begin{equation}\label{partition-min}
\textstyle{\bigcup\limits_{i=1}^{N}} (A_i \symmdiff A'_i) \subset \subset W,
\end{equation}
we have
\[
E(A;W) \leq E(A'; W).
\]
\end{definition}

The symmetric difference $A_i \symmdiff A'_i$ of the sets $A_i$ and $A'_i$ is defined as their union minus their intersection, that is, $A_i \symmdiff A'_i = (A_i \cup A'_i) \setminus (A_i \cap A'_i)$.

\subsection{Flat chains with coefficients in a group}
Let $G$ be an abelian group with norm $|\cdot|$, such that $|g| \geq 0$, with $|g|=0$ if and only if $g=0$, for all $g \in G$, and
\[ |g+h| \leq |g| + |h|, \text{ for all } g, h \in G.\]
Then, $(G, |\cdot|)$ is a metric space and we will assume that it is complete and separable. In our case $G$ will be a finite group.

Fix $\R^n$ and a compact convex set $\K$ in $\R^n$. For each integer $k \geq 0$ consider the abelian group of all formal finite sums of the form $\sum g_i P_i$, where $g_i \in G$ and where $P_i$ is a $k$-dimensional oriented compact convex polyhedron in $\K$. We form the quotient group obtained by identifying $gP$ with $-g \tilde P$, whenever $P$ and $\tilde P$ coincide but have opposite orientations. Also, identify $gP$ and $g P_1 + g P_2$, whenever $P$ can be subdivided into $P_1$ and $P_2$.

The resulting abelian group $\mathcal{P}_k (\K; G)$ is called the group of \emph{polyhedral $k$-chains} on $\K$ with coefficients in $G$. Define the \emph{boundary homomorphism} $\partial : \mathcal{P}_k \to \mathcal{P}_{k-1}$ by
\[ \partial \left( \sum g_i P_i \right) := \sum g_i \partial P_i .\]
Note that any polyhedral $k$-chain $T$ can be written as a linear combination $\sum_i g_i [P_i]$ of nonoverlapping polyhedra, that is, polyhedra with disjoint interiors. Then, the \emph{flat norm} of the chain is defined to be
\[
W(T)= \inf_Q \{ \M(T-\partial Q)+\M(Q) \},
\]
where the infinimum is over all polyhedral $(k+1)$-chains $Q$.

The flat norm makes $\mathcal{P}_k(\K; G)$ into a metric space. The completion of this metric space is denoted by $\mathcal{F}_k(\K; G)$ and its elements are called \emph{flat $k$-chains} in $\K$ with coefficients in $G$. By uniform continuity, functionals such as the flat norm and operations such as addition and boundary extend in a unique way from polyhedral chains to flat chains. The mass norm in $\mathcal{P}_k(\K; G)$ extends to a linear semicontinuous functional in $\mathcal{F}_k(\K;G)$.

Suppose that every bounded closed subset of $G$ is compact. A fundamental compactness theorem for flat chains asserts that, given any sequence $T_i \in \mathcal{F}_k(\K; G)$ with $\M(T_i)$ and $\M(\partial T_i)$ uniformly bounded, there is a $W$-convergent subsequence. More generally, one can define the flat chains in $\R^n$ with compact support, $\mathcal{F}_k(\R^n;G)$, meaning that each element vanishes outside a certain compact convex set. Then, the compactness theorem holds for a sequence $T_i \in \mathcal{F}_k(\R^n; G)$, with $\supp T_i \subset \K$, for $\K$ independent of $i$, such that $T_i\rightharpoonup T$, where the symbol `$\rightharpoonup$' denotes convergence in the flat norm.

\subsection{Flat chains of top dimension}[See \cite{white}.]
\label{flat-chains-of-top-dimension}
Polyhedral $n$-chains in $\R^n$ with compact support can be identified with the set of piecewise-constant functions
\[ g: \R^n \to G ,\]
that vanish outside a compact convex set $\K$. Here, two functions that differ only on a set of measure zero are regarded as the same. `Piecewise constant' means locally constant except along a finite collection of hyperplanes. The identification is as follows. Any such $ T \in \mathcal{F}_k(\R^n;G)$ can be written as
\[ T=\sum g_i [P_i], \]
where the $P_i$'s are nonoverlapping and inherit their orientations from $\R^n$. We can associate to $T$ the function
\[ g: \R^n \to G, \text{ with } g(x)=
\begin{cases}
g_i, &\text{if $x$ is in the interior of $P_i$ },\smallskip\\
0, &\text{if $x$ is not in the interior of $P_i$. }
\end{cases} \]
Note that the mass norm of $T$ is equal to the $L^1$ norm of $g(\cdot)$. Also, since there are no nonzero $(n+1)$-chains in $\R^n$, we see from the definition of $W$ that
\[ W(T)=\M(T)=\int_{\R^n}|g(x)| \dd x. \]
Consequently, the $W$-completion of the polyhedral chains (that is, the flat $n$-chains) is isomorphic to the $L^1$-completion of the piecewise-constant functions.

Denoting $T$ by $[\K]_{Lg}$, the isomorphism is
\[ L^1(\K;G) \ni g \to [\K]_{Lg} \in \mathcal{F}_n(\K;G), \]
with
\[ \M([\K]_{Lg})=W([\K]_{Lg})=\int_{\R^n}|g| \dd x.\]
Thus, the flat $n$-chains $T$ on $\R^n$ with compact support can be identified with the
$ L^1_{\rm loc}(\R^n;G)$ functions. The flat chains with $\M(\partial T)<+\infty$ correspond to the sets with finite perimeter (Caccioppoli sets). The $BV$ norm of the function $g$ above gives the perimeter, that is,
\[ \left\| g \right\|_{BV}=\M(\partial T). \]
The compactness for flat chains with
\[ \M(T_n)+\M(\partial T_n)<C \]
is equivalent \emph{in this setup} to the compactness of the embedding
\[ BV(\Omega) \subset \subset L^1(\Omega), \text{ for $\Omega$ bounded.} \]
The lower semicontinuity of $\M(\partial T)$ with respect to the $W$-norm is equivalent to the lower semicontinuity of the $BV$ norm with respect to $L^1$.

\subsection{The group of surface tension coefficients}[See \cite{white}.]
\label{the-group-of-surface-tension-coefficients}
The purpose next is the introduction of an appropriate group $G$ so that for the flat chain $T=\sum g_i P_i$, where $P_i=A_i$, with $A=\{ A_i \}$ a partition of $U$, there holds
\begin{equation}\label{mass-t-w}
\M(\partial T \corner W)=E(A;W).
\end{equation}
First, assume that
\begin{equation}\label{inequality-e_ij}
e_{ik}\leq e_{ij}+e_{jk}, \text{ for all } i,j,k.
\end{equation}
Let $G$ be the free $\Z_2$-module with $N$ generators $f_1,\ldots,f_N$ (one for each phase). White \cite{white} defines a norm in this group such that
\[ |f_i-f_j|=e_{ij}, \]
and the $\Z_2$-module identifies
\[ f_{i_1}-f_{j_1}=f_{i_1}+f_{j_1}. \]
Utilizing this, it is easy to see in calculating $\partial T$, and $\M(\partial T)$, that \eqref{mass-t-w} holds. In this setup, given a partition of $U$ into $N$ measurable sets $A_1,\ldots,A_N$, and $\K$ as above, we associate the flat $n$-chain
\[ T=\K_{Lg}, \]
where
\[ g(x)=
\begin{cases}
f_i, &\text{for } x \in A_i \cap \K \smallskip \\
0, &\text{for } x \notin A_i \cap \K.
\end{cases}
\]
Note that if the $A_i$'s have piecewise-smooth boundaries, then \eqref{mass-t-w} holds. More generally, equation \eqref{mass-t-w} holds whenever the $A_i$'s are Caccioppoli sets, that is, whenever the flat chains have finite mass.

Conversely, given any flat $n$-chain $T$, we can represent $T$ as
\[ T=\K_{Lg}, \]
where $g \in L^1(U \cap \K; G)$. In this article we take $U=\R^n$. We note that in \cite{white} it is shown that the inequalities \eqref{inequality-e_ij} are no real restriction, in the sense that if they are violated, then one can define new coefficients $e^{*}_{ij}$ out of the old, so that the infimum of $E$ coincides with the infimum of $E^{*}$ (defined by replacing $e_{ij}$ with $e^{*}_{ij}$). Also, it is noted that \eqref{inequality-e_ij} is necessary for $E$ to be lower semicontinuous with respect to the flat norm. Here, we refer also to \S 4.1 in \cite{alberti}.

\subsection{Basics on minimizing chains}
\label{basics-on-minimizing-chains}
We recall some standard facts on minimizing chains and later we point out the relationship with minimizing partitions.

\subsubsection*{Cones}
If $x_0 \in \R^n$, where $S$ is a $k$-dimensional flat chain in $\R^n$, then \emph{the cone over $S$ with vertex at $x_0$} is the flat chain
\begin{equation}\label{definition-cone}
 x_0 S = \mathrm{Cone}(S) = h(I \times S),
\end{equation}
where $h(t,x)=(1-t)x_0+tx$, for $0 \leq t \leq 1$, and $x \in S$. We have
\[ S = \partial(x_0 S) + x_0 \partial S, \]
and if $S \subset B_r(x_0)$, where $B_r(x_0)$ is the ball with radius $r$ and center at $x_0$, then
\[ \M( x_0\, S) \leq \frac{r}{k+1}\M(S). \]

$C_x$ is a cone with vertex at $x$ if, by definition, it is invariant as a set under the homothetic map
\[ y \to x+t(y-x), \text{ for all } t>0 \text{ and } y \in C_x. \]
If $S$ is a $k$-flat chain in $\R^n$, then $S$ is \emph{mass minimizing} if $\M(S) \leq \M(S')$, for all $S'$ with $\partial S' = \partial S$.

If $\Gamma$ is a \emph{cycle}, that is, $\partial \Gamma =0$, then
\[ L(\Gamma) := \inf \{ \M(X) \mid \partial X=\Gamma \}. \]
Consequently, if $S$ is mass minimizing, then $\M(S)=L(\partial S)$.

The flat chain $S$ is \emph{minimizing} if by definition
\[ \M(S \corner B_r(x)) = L(\partial (S \corner B_r(x)), \]
for all $r>0$ and center $x$ such that $0<r<\dist(x, \supp\partial S)$. Mass minimizing is minimizing. We allow the options $\partial S=0$ and $\M(S)=\infty$.

The \emph{monotonicity formula} holds for $k$-dimensional minimizing flat chains and states that
\begin{equation}\label{monotonicity-cones}
\Theta(S,x,r) := \frac{\M(S \corner B_r(x))}{\omega_k r^k}
\end{equation}
is an increasing function of $r$, where $\omega_k$ is the volume of the $k$-dimensional unit ball. It follows that for minimizing flat chains $S$, the limit
\[ \Theta(S,x) := \lim_{r \to 0} \Theta(S,x,r) \]
exists, and if
\begin{equation}\label{bounded-theta}
\Theta(S,x,r) < B, \text{ for } B \text{ independent of } x, r,
\end{equation}
then the limit
\[ \Theta(S) := \lim_{r \to +\infty} \Theta(S,x,r) \]
exists and is independent of $x$. We note that the condition $\Theta(S,x,r)=\const$ in $r>0$, for $x$ fixed, implies that $S$ is a cone with vertex at $x$.

\subsubsection*{The tangent cone (blow-up)}
Let $S$ be a minimizing flat chain, $x \notin \supp \partial S$, and let $\{ \mu_i \}$ be an increasing sequence of positive numbers, with $\mu_i \to +\infty$. Set
\[ S_i=\mathcal{D}_{\mu_i}(S-x), \text{ with } \mathcal{D}_{\mu_i}(S-x)=\{ \mu_i (y-x) \mid y \in S \}. \]
Then \emph{along a subsequence} there holds $S_i \rightharpoonup C_x$ (by the compactness theorem), where $C_x$ has the properties
\begin{enumerate}
\item $\partial C_x=0$,
\item $C_x$ is a cone,
\item $\Theta (C,0,r)=\Theta (S,x)$, for all $r>0$,
\item $C_x$ is minimizing.
\end{enumerate}

\subsubsection*{The cone at infinity (blow-down)}
If instead in the arrangement above $\{ \mu_i \}$ is a decreasing sequence, with $\mu_i \to 0$, and if \eqref{bounded-theta} holds, then along a subsequence there holds $S_i \rightharpoonup C_{\infty}$ (by the compactness theorem), where $C_{\infty}$ has the properties
\begin{enumerate}
\item $\partial C_{\infty}=0$,
\item $C_{\infty}$ is a cone,
\item $\Theta (C_{\infty},0,r)=\Theta (C_{\infty},0)=:\Theta(S)$,
\item $C_{\infty}$ is minimizing.
\end{enumerate}
Note that $C_x$ and $C_{\infty}$ are not necessarily unique.

If $N^{k-1}$ is a smooth $(k-1)$-surface, with $k \leq n-1$ and $N^{k-1} \subset \mathbb{S}^{n-1}$ (the unit sphere in $\R^n$), then the \emph{cone over $N^{k-1}$} is
\[ C(N^{k-1}) = \left\{ x \in \R^n ~\bigg|~ \frac{x}{|x|} \in N^{k-1} \right\}. \]
If $S$ is smooth, then the projection
\[ \frac{1}{R}(S \cap \mathbb{S}^{n-1}_R), \]
of the set $S \cap \mathbb{S}^{n-1}_R$ on the unit sphere tends to $C_{\infty}$, as $R \to \infty$, provided that \eqref{bounded-theta} holds, that is,
\[ C_{\infty} = \lim_{R \to +\infty} C \left( \frac{1}{R} \left( S \cap \mathbb{S}^{n-1}_R \right) \right), \]
where the limit is in the flat norm, and exists along a sequence
\[ R_1<R_2< \cdots \to +\infty,\]
where $C_{\infty}$ is the cone at infinity.

\subsubsection*{Relationship with partitions}[See \cite{almgren}, \cite{simon2}, \cite{chan}.]
The concepts in Paragraph \ref{basics-on-minimizing-chains} have exact analogs for partitions defined as flat chains of top dimension in Paragraphs \ref{flat-chains-of-top-dimension} and \ref{the-group-of-surface-tension-coefficients} above. Specifically,
\begin{enumerate}
\item the concept of the cone is unchanged,
\item the mass minimizing flat chain $S$ is replaced by the minimizing partition $T$ (or $A$), via the definition in \eqref{partition-min} above,
\item the monotonicity formula holds for minimizing partitions,
\item the notion of tangent cone and cone at infinity have exact analogs for minimizing partitions.
\end{enumerate}

\section{A related Bernstein-type theorem -- Statements and proofs}
\label{section5}
We are now ready to state a sample of a Bernstein-type theorem. We begin with $\R^2$.
\begin{theorem}[$n=2$]
\label{bernstein-theorem-1} 
Let $A$ be a complete minimizing partition in $\R^2$ with $N=3$ (three phases), with surface tension coefficients satisfying
\begin{equation}\label{inequality-e_ij-strict}
e_{ik} < e_{ij}+e_{jk}, \text{ for } j \neq i,k \text{ with } i,j,k \in \{1,2,3\}.
\end{equation}
Then, $\partial A$ is a triod.
\end{theorem}

\begin{figure}[t]
\begin{center}
\begin{picture}(0,0)%
\includegraphics{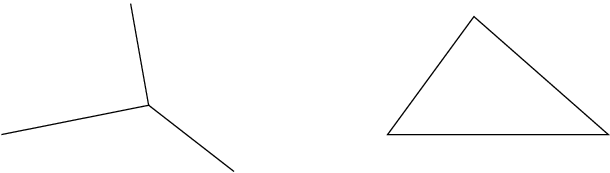}%
\end{picture}%
\setlength{\unitlength}{4144sp}%
\begin{picture}(4648,1302)(-236,434)
\put(665,1050){$\theta_2$}%
\put(2808,1289){$e_{23}$}%
\put(3877,1289){$e_{12}$}%
\put(3427,550){$e_{13}$}%
\put(3990,810){$\hat{\theta}_1$}%
\put(2900,810){$\hat{\theta}_3$}%
\put(3330,1380){$\hat{\theta}_2$}%
\put(798,748){$\theta_3$}%
\put(980,999){$\theta_1$}%
\put(1300,1250){$1$}%
\put(200,1200){$2$}%
\put(700,400){$3$}%
\end{picture}%
\caption{}
\label{figure1}
\end{center}
\end{figure}

In Figure \ref{figure1} we show a triod with angles $\theta_1$, $\theta_2$, $\theta_3$, and the corresponding triangle with their supplementary angles $\hat{\theta}_i = \pi - \theta_i$. For these angles Young's law holds, that is, 
\[ \frac{\sin \hat{\theta}_1}{e_{23}} = \frac{\sin \hat{\theta}_2}{e_{13}} = \frac{\sin \hat{\theta}_3}{e_{12}}. \]

We recall that under the condition of the strict triangle inequality for the surface tension coefficients, White has established a general regularity result which applies in particular under \eqref{inequality-e_ij-strict} to $A$ above. His result improves on Almgren's work \cite{almgren}. Detailed proofs can be found in Chan's thesis \cite{chan} (\S 1.6 and pp.\ 10--14). It follows that $A$ consists of triple junctions and line segments, always a finite number in any given open and bounded subset of $\R^2$.

We present the proof of Theorem \ref{bernstein-theorem-1} in three steps. The first two are two lemmas that we state next.

\begin{lemma}\label{only-minimizing-cones}
The only minimizing cones are the straight line and the triod.
\end{lemma}

\begin{lemma}\label{inequality-mass}
There holds $\M(\partial A \corner B_R) \leq CR$.
\end{lemma}

Accepting for the time being the lemmas above, we can conclude with the proof of the theorem.

\begin{proof}[Proof of Theorem \ref{bernstein-theorem-1}]
Let $P$ be one of the junctions and consider the tangent cone $C_0$ at $P$, which by Lemma \ref{only-minimizing-cones} is a triod. We have
\begin{align}\label{theta-c_0}
\Theta(C_0,P)=\Theta(A,P)&\leq \frac{\M(A \corner B_R(P))}{\pi R^2} \leq \lim_{R \to +\infty}\frac{\M(A \corner B_R(P))}{\pi R^2} \leq \Theta (A),
\end{align}
where we used the monotonicity formula \eqref{monotonicity-cones} and also the bound provided by Lemma \ref{inequality-mass} above.

Let $C_{\infty}$ be any of the cones at infinity. Then
\begin{equation}\label{inequality-c_infty}
1<\Theta(C_0,P) \leq \Theta(A)=\Theta(C_{\infty}).
\end{equation}
From Lemma \ref{only-minimizing-cones} we conclude that
\begin{equation}\label{c_0=c_infty}
C_0=C_{\infty}, \text{ up to congruence.}
\end{equation}
Returning back to \eqref{theta-c_0} and utilizing \eqref{c_0=c_infty}, we obtain
\begin{equation}\label{mass-theta-a}
\frac{\M(A \corner B_R(P))}{\pi R^2} = \Theta (A), \text{ for all } R>0.
\end{equation}
It follows from \eqref{mass-theta-a} that $A$ is a cone and from the first inequality in \eqref{inequality-c_infty} that it is a singular one, hence a triod by Lemma \ref{only-minimizing-cones}.
\end{proof}

Note that the scheme above for proving a Bernstein-type theorem was introduced by Fleming \cite{fleming}, where it was used to give a new proof of the classical Bernstein theorem in $\R^3$. (See p.\ 193 in \cite{colding-minicozzi}.)

We now give the proofs of Lemmas \ref{only-minimizing-cones} and \ref{inequality-mass}.

\begin{proof}[Proof of Lemma \ref{only-minimizing-cones}]
We begin by recalling the classical Steiner problem adapted in our weighted setup. Let $A$, $B$, $C$, be three fixed points on the plane. Given a fourth point $\hat P$, not necessarily distinct from the other three, consider the weighted sum of the distances of $\hat P$ from the vertices of the triangle $ABC$, that is,
\begin{equation}\label{PABC}
 e_{12}|\hat P -A|+e_{23}|\hat P -C|+e_{13}|\hat P -B|.
\end{equation}
Suppose that the quantity above is minimized for $P \not= A,B,C$. 

\begin{figure}[t]
\begin{center}
\begin{picture}(0,0)%
\includegraphics{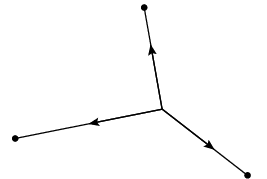}%
\end{picture}%
\setlength{\unitlength}{4144sp}%
\begin{picture}(2019,1353)(-339,419)
\put(1300,713){$\nu_{13}$}
\put(290,705){$\nu_{23}$}
\put(550,1388){$\nu_{12}$}
\put(980,999){$\theta_1$}
\put(665,1050){$\theta_2$}
\put(798,748){$\theta_3$}
\put(820,1650){$A$}
\put(1534,525){$B$}
\put(-250,810){$C$}
\end{picture}
\caption{}
%\caption{A triod and the corresponding Neumann's triangle.}
\label{figure2}
\end{center}
\end{figure}

We recall Steiner's argument. Let
\[ \Gamma=\{ Q \in \R^2 \mid F(Q)=F(P) \}, \]
where
\[ F(Q)=e_{12}|Q -A|+e_{13}|Q -B|. \]
Then, one notes that necessarily the circle with center $C$ and radius $|P-C|$ has to be tangent to the curve $\Gamma$ at $P$. Therefore, $P$ in particular solves the problem
\[ \min F(Q), \text{ subject to } L(Q):=\langle Q, \nu_{23} \rangle=0, \]
from which it follows that
\[ e_{12} \nu_{12} + e_{13} \nu_{13}=\alpha \nu_{23}, \text{ for some } \alpha \in \R. \]
By replacing $A$, $B$ with $B$, $C$ first, and then with $C$, $A$, and repeating the argument in each of these cases we obtain respectively
\begin{align*}
e_{13} \nu_{13} + e_{23} \nu_{23} = \beta \nu_{12}, \text{ for some } \beta \in \R,\\
e_{12} \nu_{12} + e_{23} \nu_{23} = \gamma \nu_{13}, \text{ for some } \gamma \in \R.
\end{align*}
From these relationships it follows easily that
\[ e_{12} \nu_{12} + e_{13} \nu_{13} + e_{23} \nu_{23} = 0, \]
and thus Young's law is established.

We explain under which conditions the quantity in \eqref{PABC} is minimized for a $P \neq A,B,C$. It is clear that for the partitioning in Figure \ref{figure2} to be possible, one needs that all the inequalities $\hat A < (\pi -\theta_1)+(\pi -\theta_2)$, $\hat B < (\pi -\theta_1)+(\pi -\theta_3)$, and $\hat C < (\pi -\theta_2)+(\pi -\theta_3)$ should hold, where $\hat{\theta}$ stands for $\pi - \theta$. Thus, in general one needs that there is an arrangement so that all inequalities $\hat A < (\pi -\theta_i)+(\pi -\theta_j)$, $\hat B < (\pi -\theta_i)+(\pi -\theta_k)$, and $\hat C < (\pi -\theta_j)+(\pi -\theta_k)$ hold, where $i,j,k \in \{1,2,3 \}$, and $\theta_1$, $\theta_2$, $\theta_3$ are determined by the triangle in Figure \ref{figure1}. For example, if all surface tension coefficients are equal, then $P \neq A,B,C$ if and only if the largest angle of $ABC$ is less than $\frac{2 \pi}{3}$.

Next, we show the connectedness of each phase in a minimizing cone $C$. That is, each phase can appear only once (cf.\ Lemma~1.2 in \cite{chan}). We proceed by contradiction. First we exclude the possibility that one phase is separated by another. By regularity, the cone consists of finitely many rays emanating from the origin $O$. Suppose that we have the arrangement as in Figure \ref{figure3a}.

\begin{figure}[t]
\begin{subfigure}{0.4\textwidth}
\centering
\begin{picture}(0,0)%
\includegraphics{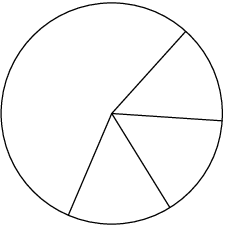}%
\end{picture}%
\setlength{\unitlength}{4144sp}%
\begin{picture}(1887,1860)(3537,-732)
\put(4894,447){$2$}%
\put(4219,305){$O$}%
\put(3910,-678){$A$}%
\put(4838,-622){$B$}%
\put(5288,165){$C$}%
\put(5020,915){$D$}%
\put(4387,-342){$2$}%
\put(4866,-88){$3$}%
\end{picture}%
\subcaption{}
\bigskip
\bigskip
\bigskip
\label{figure3a}
\end{subfigure}
\quad
\begin{subfigure}{0.4\textwidth}
\centering
\begin{picture}(0,0)%
\includegraphics{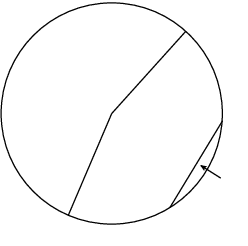}%
\end{picture}%
\setlength{\unitlength}{4144sp}%
\begin{picture}(1887,1860)(3537,-732)
\put(4219,305){$O$}%
\put(3910,-678){$A$}%
\put(4838,-622){$B$}%
\put(5288,165){$C$}%
\put(5020,915){$D$}%
\put(4700,50){$2$}%
\put(5250,-270){$3$}%
\end{picture}%
\subcaption{}
\bigskip
\bigskip
\bigskip
\label{figure3b}
\end{subfigure}
\begin{subfigure}{0.4\textwidth}
\centering
\begin{picture}(0,0)%
\includegraphics{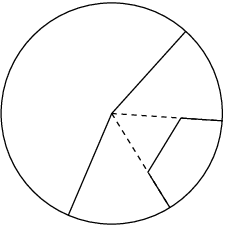}%
\end{picture}%
\setlength{\unitlength}{4144sp}%
\begin{picture}(1887,1860)(3537,-732)
\put(4613, 60){$2$}%
\put(4219,305){$O$}%
\put(3910,-678){$A$}%
\put(4838,-622){$B$}%
\put(5288,165){$C$}%
\put(5020,915){$D$}%
\put(4908,-102){$3$}%
\end{picture}%
\subcaption{}
\label{figure3c}
\end{subfigure}
\quad
\begin{subfigure}{0.4\textwidth}
\centering
\begin{picture}(0,0)%
\includegraphics{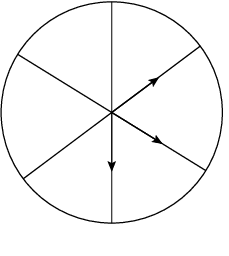}%
\end{picture}%
\setlength{\unitlength}{4144sp}%
\begin{picture}(1746,1949)(3537,-817)
\put(5007,265){$2$}%
\put(4480,590){$\nu_{12}$}%
\put(4778,115){$\nu_{23}$}%
\put(4455,-117){$\nu_{13}$}%
\put(4065,-327){$1$}%
\put(4628,785){$1$}%
\put(4057,783){$3$}%
\put(3699,223){$2$}%
\put(5119,784){$A$}%
\put(5147,-257){$B$}%
\put(4304,-763){$C$}%
\put(4627,-327){$3$}%
\end{picture}%
\subcaption{}
\label{figure3d}
\end{subfigure}
\label{figure3}
\caption{}
\end{figure}

We will exclude the case in Figure \ref{figure3a} by comparing with Figure \ref{figure3b}. The cone induces a partition $A$ of the unit disk $D$. Clearly, $\M(\partial A) > \M(\partial A')$, where $A'$ is the partition in Figure \ref{figure3b}, and since $A \symmdiff A' \subset \subset D$, we obtain a contradiction, since $C$ was assumed minimizing. One may object to whether $A \symmdiff A' \subset \subset D$ is satisfied. For this purpose, we can modify $A'$ as in Figure \ref{figure3c}. Thus the only remaining possibility is that the disk is partitioned by a number of triple junctions as in Figure \ref{figure3d}.

We will show that only a single triple junction is acceptable. Consider the triangle $ABC$. By the hypothesis, $C$ is a minimizing cone. Hence $O$ solves the Steiner problem with respect to this triangle. But then, the angles $\theta_1$, $\theta_2$, $\theta_3$ have to satisfy Young's law, and in particular have to add up to $2 \pi$, which completes the proof of the lemma.
\end{proof}

\begin{remark}
We mention below a more general and also simpler argument due to the referee, which establishes the connectedness of each phase. Note, however, that this argument is no substitute for Steiner's since obviously the center of a triod is not in general equidistant from the three points. That the admissible cones are made up of a finite number of rays follows from regularity.

Let $\nu_1,\dots,\nu_N$, with $\nu_{N+1} = \nu_1$, ordered counterclockwise around $O$, be the rays (unit vectors) of the unit disk $D$ which determine the minimizing cone $A$ and let $e_1,\dots,\e_N$ be the corresponding surface tension coefficients. Given $i<j$, we denote by $\nu_i \nu_j \subset D$ the sector swept by a unit vector $\nu$ that rotates in the positive direction from $\nu_i$ to $\nu_j$ around $O$. Since by assumption $A$ has three phases, if $N \geq 4$ there exist $i \neq j$ such that the sectors $\nu_i \nu_{i+1}$ and $\nu_j \nu_{j+1}$ are contained in the same phase. One at least of the sectors $\nu_{i+1} \nu_{j}$ and $\nu_{j+1} \nu_{i}$, say $\nu_{i+1} \nu_{j}$, has angle smaller than $\pi$. Set $\nu = (\nu_{i+1} + \nu_{j})/2$ and, for small positive $\varepsilon$, let $A_\varepsilon$ be the partition defined by displacing the first extreme $O$ of the rays $\nu_{i+1}, \nu_{i+2}, \dots, \nu_{j}$ to $O + \varepsilon \nu$. A simple computation yields 
\[
\frac{\dd}{\dd \varepsilon} \Big|_{\varepsilon = 0} \left( \sum_{k=i+1}^{j} e_k |\nu_k + \varepsilon \nu| \right) = \nu \cdot\! \sum_{k=i+1}^{j} e_k \nu_k < 0,
\]
in contradiction with the minimality of $A$. This establishes that $N=3$.

To prove Young's law we let $A_\varepsilon$ be the partition defined by displacing the first extreme of all the rays to $O + \varepsilon \nu$, where $\nu$ is an arbitrary vector. In this case, we obtain
\[
\frac{\dd}{\dd \varepsilon} \Big|_{\varepsilon = 0} \left( \sum_{k=i+1}^{j} e_k |\nu_k + \varepsilon \nu| \right) = \nu \cdot\! \sum_{k=i+1}^{j} e_k \nu_k.
\]
This, the minimality of $A$, and the arbitrary choice of $\nu$ imply that
\[
\sum_{k=1}^{N} e_{k} \nu_{k} = 0,
\]
and since $N=3$, it follows that $e_1 \nu_1 + e_2 \nu_2 + e_3 \nu_3 = 0$.
\end{remark}

\begin{proof}[Proof of Lemma \ref{inequality-mass}]
Consider the disk $B(O;R)$. By the regularity of $A$, the intersection 
\[ \partial B(O;R) \cap \partial A \] 
consists of finitely many points, say $k$. We now construct a test partition $\tilde A$ as follows. First, enlarge the $R$-disk slightly so that $\M(\partial A \corner (B_{R'} \setminus B_R)) < \varepsilon$, for $R'>R$, with $R'-R \ll 1$. This is possible by the regularity of $A$. Take $\partial \tilde A = \partial A$ inside the ring $R \leq |x| \leq R'$. Next, we introduce an $\varepsilon^2$-layer inside $B(O;R)$, in which we take $\partial \tilde A$ to consist of the union of $k$ $\varepsilon^2$-line segments orthogonal to $\partial B(O;R)$, emanating from the $k$ points on $\partial B(O;R)$. Finally, inside $B(O;R-\varepsilon^2)$ take a single phase, say phase $1$. (See Figure \ref{figure4b}.)

\begin{figure}[t]
\begin{subfigure}{0.4\textwidth}
\centering
\begin{picture}(0,0)%
\includegraphics{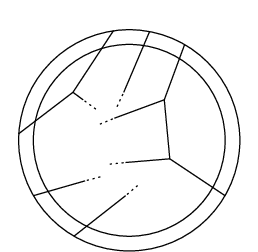}%
\end{picture}%
\setlength{\unitlength}{4144sp}%
\begin{picture}(0,0)%
\includegraphics{disk5.ps}%
\end{picture}%
\setlength{\unitlength}{4144sp}%
\begin{picture}(1959,1910)(3403,-574)
\put(3572,867){$1$}%
\put(3657,-455){$1$}%
\put(4669,1148){$2$}%
\put(4360,1190){$3$}%
\put(4247,811){$3$}%
\put(4444,249){$3$}%
\put(4528,755){$2$}%
\put(3797,164){$2$}%
\put(3403, -5){$2$}%
\put(3995,-230){$1$}%
\put(4866,361){$1$}%
\put(5260,530){$1$}%
\end{picture}%
\subcaption{}
\label{figure4a}
\end{subfigure}
\quad
\begin{subfigure}{0.4\textwidth}
\centering
\begin{picture}(0,0)%
\includegraphics{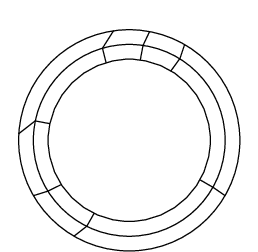}%
\end{picture}%
\setlength{\unitlength}{4144sp}%
\begin{picture}(1959,1910)(3403,-574)
\put(3657,-455){$1$}%
\put(4360,249){$1$}%
\put(3572,867){$1$}%
\put(5260,530){$1$}%
\put(3403, -5){$2$}%
\put(4360,1190){$3$}%
\put(4669,1148){$2$}%
\end{picture}%
\subcaption{}
\label{figure4b}
\end{subfigure}
\label{figure4}
\caption{}
\end{figure}

By construction, $A$ and $\tilde A$ have the same Dirichlet values in $B(O;R')$. We have
\begin{align*}
\M(\partial \tilde A \corner B_{R'})&=\M(\partial \tilde A \corner (B_{R'} \setminus B_R))+\M(\partial \tilde A \corner B_R) \\
&\leq \varepsilon +\M(\partial \tilde A \corner B_R) \\
&\leq \varepsilon +k\varepsilon^2 + 2 \pi (R -\varepsilon^2).
\end{align*}
Taking now $\varepsilon=\frac{1}{k}$, we obtain $\M(\partial \tilde A \corner B_{R'})\leq CR$, and thus, since $A$ is minimizing, the estimate of the lemma follows.
\end{proof}

\begin{remark}
Suppose $C$ is an $(n-1)$-dimensional cone embedded in $\R^n$, $3 \leq n \leq 7$, with zero mean curvature, which is also stable, then $C$ is a hyperplane. This fails for $n \geq 8$. This is a classical result due to Simons (see Appendix B in \cite{simon1}). The difference with the cones considered here lies in orientation. In the oriented case above there is a globally well-defined unit normal and stability is understood in the class of normal perturbations. The second variation gives the stability inequality
\[ \int_C \left\{ |\nabla^C z|^2-|A|^2 z^2 \right\} \geq 0,\]
where $|A|$ is the second fundamental form of $C$. This is not the appropriate formula for the class of unoriented objects we consider here. The perturbations in the unoriented case have also tangential components. Another difference is in the growth estimate 
\[ \M(B_r \cap \Sigma) \leq Cr^{n-1},\] 
which is immediate for minimizing complete hypersurfaces in $\R^n$ (see p.~4 in \cite{colding-minicozzi}).
\end{remark}

Our purpose next is to establish the following theorem.

\begin{theorem}[$n=3$]\label{bernstein-theorem-2} 
Let $A$ be a complete minimizing partition in $\R^3$ with $N=3$ (three phases), with surface tension coefficients satisfying
\[ e_{ik} < e_{ij}+e_{jk}, \text{ for } j \neq i,k \text{ with } i,j,k \in \{1,2,3\}, \]
and with the property that the liquid edges do not intersect. Then,
\begin{equation}\label{cylindrical-cone}
\partial A= \R \times C_\mathrm{tr} \quad \text{(cylindrical cone),}
\end{equation}
where $C_\mathrm{tr}$ is the triod on the plane.
\end{theorem}

We note that the subset of $\partial A$ whose tangent cones contain exactly three phases is the union of the liquid edges, denoted by $\Sigma_3(A)$. We recall that in the case of equal surface tension coefficients the cone in \eqref{cylindrical-cone} represents one of the two stable types of singularities that soap films can form. This was shown by Taylor \cite{taylor} who also proved $(e_{ij}=1)$ that if $A$ is a minimizing partition (not necessarily complete), then $\Sigma_3(A)$ consists of a union of $C^{1,\alpha}$ curves. Furthermore, Kinderlehrer, Nirenberg, and Spruck \cite{kinderlehrer-nirenberg-spruck} showed that these curves are real analytic. 

The basis for Theorem \ref{bernstein-theorem-2} above is the following result of Chan \cite{chan}.

\begin{theorem}[\cite{chan}]\label{theorem Chan}
Let $A$ be a complete stable partition in $\R^3$ with $N=3$ (three phases), where $\Sigma_3(A)=\cup_i \gamma_i$, with $\gamma_i$ a smooth curve, and suppose that the density at infinity $\Theta(A)$ is finite. Then, $A \setminus \Sigma_3$ is a union of planar pieces.
\end{theorem}

A smooth partition is \emph{stable} if the second variation is nonnegative for $C^2$ compactly supported perturbations $Y$, that is,
\begin{multline*}
\frac{\dd^2}{\dd t^2} \Big|_{t=0} \M(M+tY) = \int_M (-|B|^2(Y^{\bot})-Y^{\bot}\! \cdot \Delta_M (Y^{\bot}) ) \\
\quad-\int_{\partial M}(\dv_M Y^{\top}(Y^{\top}\! \cdot \mu)+[Y^{\bot},Y^{\bot}] \cdot \mu ) \geq 0,
\end{multline*}
where $M$ is a regular surface immersed in $\R^3$, the vector field $Y$ is defined in a neighborhood of $M$, $\mu$ is the unit outward conormal on $\partial M$, $Y^{\bot}$ and $Y^{\top}$ are the normal and tangential components of $Y$ respectively, and $|B|$ is the length of the second fundamental form on $M$ (see \cite{spivak}).

By rearranging the stability inequality above one obtains control of the second fundamental form by terms involving the test variations $Y$. One of the important ingredients in the proof is the choice of the variations $Y$ at the junctions, where the Young angle conditions are utilized for the matching. The argument finally utilizes a logarithmic cut-off function, that is, the `logarithmic trick' (see p.~30 in \cite{colding-minicozzi}), and renders $|B|=0$. We note that Chan's result as it stands does not exclude the possibility, for example, of two triple-junctions (see Figure \ref{figure5a}).

In Theorem \ref{bernstein-theorem-2} we assume that the partition is minimizing, a much stronger condition than stability. The smoothness follows by the regularity results of White mentioned after the statement of Theorem \ref{bernstein-theorem-1} above.

The proof of Theorem \ref{bernstein-theorem-2} is based on the following lemma.

\begin{lemma}\label{inequality-mass-2}
There holds $\M(\partial A \corner B_R) \leq CR^2$.
\end{lemma}

\begin{proof}[Proof of Theorem \ref{bernstein-theorem-2}]
We employ a dimension reduction argument. Accepting for the moment the estimate in Lemma \ref{inequality-mass-2}, we conclude as follows. By regularity facts (see the proof of Lemma \ref{inequality-mass-2} below), Chan's theorem above applies and gives that $A \setminus \Sigma_3$ is a union of planar pieces. By the hypothesis that the liquid edges (now straight lines) do not intersect, $A$ has cylindrical structure, $A=A_2 \times\R$, where $A_2$ is a complete minimizing partition in $\R^2$. Thus, by Theorem \ref{bernstein-theorem-1} we are set.
\end{proof}

We note that it is easy to see that the cone at infinity $C_{\infty}$ for the two triple junction system in Figure \ref{figure5a} is as in Figure \ref{figure5b}, which is not minimizing.

\begin{figure}[t]
\begin{subfigure}{0.4\textwidth}
\centering
\includegraphics[scale=.5]{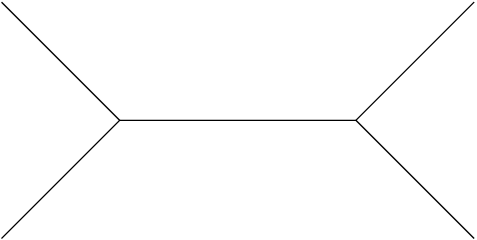}%
\subcaption{}
\label{figure5a}
\end{subfigure}
\quad
\begin{subfigure}{0.4\textwidth}
\centering
\includegraphics[scale=.5]{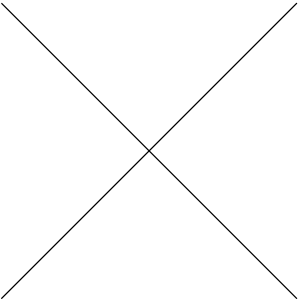}%
\subcaption{}
\label{figure5b}
\end{subfigure}
\caption{}
\label{figure5}
\end{figure}

The result of Theorem \ref{bernstein-theorem-2} states in particular that for minimizing $3$-partitions, the liquid edge is connected, or equivalently that each phase is connected. This question can be addressed in bounded domains $U \subset \R^n$. In \cite{alikakos-faliagas2}, Faliagas and the author establish connectedness for strictly convex $U$ in the class of stable partitions, thus extending to $3$-partitions the work of Sternberg and Zumbrun \cite{sternberg-zumbrun}.

\begin{proof}[Proof of Lemma \ref{inequality-mass-2}]
We begin by recalling certain regularity results. First, by White \cite{white} (see in \cite{chan}), $\partial A$ is regular except along a singular set of at most Hausdorff dimension $n-2=3-2=1$. Next, Simon \cite{simon2} (see also \cite{chan}) applies and gives that the singular set consists of $C^{1, \alpha}$ curves, since for $n=3$ and $N=3$ there are no gaps in the stratification of $\partial A$ by `spine' dimension. Finally, by Kinderlehrer, Nirenberg, and Spruck \cite{kinderlehrer-nirenberg-spruck}, these curves are actually real analytic.

The procedure for deriving the estimate is analogous to the two-dimensional case in Lemma \ref{inequality-mass}. We consider a ball $B(O;R)$ and look at the intersection $\partial B(O;R) \cap \partial A$, which consists of curves intersecting along a finite number of triple junctions (by our hypothesis that liquid edges do not intersect). Thus on $\mathbb{S}^2(O;R)$ we have a finite network made up of curves and triple junctions. We will be introducing two concentric spheres $\mathbb{S}^2(O;R_1)$ and $\mathbb{S}^2(O;R_2)$, with $R_1<R<R_2$, and a test partition $\tilde A$, which will be obtained by modifying $A$ in the annular region $\textrm{Ann}(O;R_1,R)$ in a way we explain below, and continued with a single phase in $B(O; R_1)$. Finally, $\tilde A$ will coincide with $A$ in $\textrm{Ann}(O;R,R_2)$. At the end the radii $R_1$, $R_2$ will be taken suitably close to $R$. To begin, we project radially on $\mathbb{S}^2(O;R_1)$ the network of curves and junctions $\partial B(O;R) \cap \partial A$, and also consider the surface swept out by this radial projection, whose mass can be estimated by
\[
\M(\partial \tilde A \corner \mathrm{Ann}(O;R_1,R))\leq L(R)(R-R_1),
\]
where $L(R)$ is the total length of $\partial B(O;R) \cap \partial A$, which by regularity is finite. Inside $B(O;R_1)$ we take a single phase, say phase $1$. Given $\varepsilon > 0$, by regularity we can take $R_2-R$ small enough so that
\[
\M(\partial A \corner \mathrm{Ann}(O;R,R_2)) <\varepsilon.
\]
Taking also $L(R)(R-R_1)< \varepsilon$, we obtain
\[
\M(\partial \tilde A \corner B_R)\leq \varepsilon+\varepsilon+4 \pi R_1^2 \leq 2 \varepsilon +4 \pi R^2,
\]
and since $A$ and $\tilde A$ have the same Dirichlet values in $B_R$ and $A$ is minimizing, the lemma is established.
\end{proof}

\section{The hierarchical structure of equivariant connection maps}
\label{section6}
In this section we study asymptotic properties of the equivariant solutions to system \eqref{system} produced by Theorem \ref{theorem1}. To explain the general result we have in mind, we begin with an example. Let $n=3$, $m=3$, and consider a quadruple-well potential $W$ with minima $\{ a_1, \dots, a_4 \}$, which is invariant under the group of symmetries of the regular tetrahedron, that is, $G=\mathcal{T}$. For example (see \cite{alikakos-fusco-pisa}), consider the potential
\[
W(u_1, u_2, u_3) = |u|^4 - \frac{4}{\sqrt{3}} (u^{2}_{1} - u^{2}_{2}) u_3 - \frac{2}{3} |u|^2 + \frac{5}{9}.
\]
For this choice we have
\[
N = \frac{|\mathcal{T}|}{|G_{a_1}|} = \frac{24}{6} = 4, \text{ for the minimum } a_1 = \left( \sqrt{\frac{2}{3}},\, 0,\, \frac{1}{\sqrt{3}} \right),
\]
with $D = \{ \cup gF \mid g \in G_{a_1} \}$ the simplicial cone generated by
\[
\left( 0,\, \sqrt{\frac{2}{3}},\, \frac{1}{\sqrt{3}} \right),\quad \left( 0,\, -\sqrt{\frac{2}{3}},\, \frac{1}{\sqrt{3}} \right), \quad \left( \sqrt{\frac{2}{3}},\, 0,\, -\frac{1}{\sqrt{3}} \right).
\]
Four copies of $D$ partition $\R^3$. Note that the boundary of the partition is made up of six reflection planes and coincides with the tetrahedral minimizing cone in $\R^3$ (see Figure \ref{figure6}).

\begin{figure}[t]
\begin{minipage}{0.8\textwidth}
\centering
\begin{picture}(0,0)%
\includegraphics{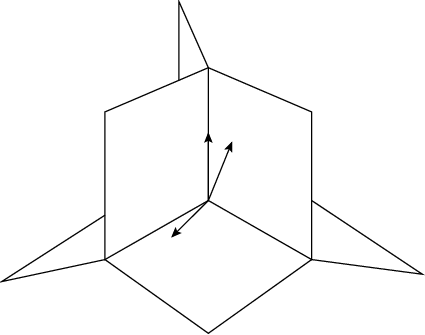}%
\end{picture}%
\setlength{\unitlength}{4144sp}%
\begin{picture}(3230,2554)(2464,-635)
\put(3090,-481){$a_{4}$}%
\put(3821,868){$\nu_{2}$}%
\put(4608,910){$\Pi_{1}$}%
\put(4299,826){$\nu_{1}$}%
\put(3617,  4){$\nu_{0}$}%
\put(3314,910){$\Pi_{2}$}%
\put(2640,699){$a_{3}$}%
\put(5284,700){$a_{2}$}%
\put(3990,-200){$a_{1}$}%
\end{picture}%
\caption{}
\label{figure6}
\end{minipage}
\end{figure}

Next, we focus on $\overline{D}$ and its walls $\Pi_1$, $\Pi_2$. Take $\nu_0 \in \interior (\overline{D}) $, $\nu_1 \in \interior (\Pi_1) $, and $\nu_2 \in \interior (\Pi_1 \cap \Pi_2) $, with $| \nu_i | = 1$, for $i=0,1,2$, and consider the limits
%\begin{equation}
\begin{align}
&\lim_{\lambda \to +\infty} u(x + \lambda \nu_0) = u_0 (P_0 x),\label{lambda-limit-1}\\
&\lim_{\lambda \to +\infty} u(x + \lambda \nu_1) = u_1 (P_1 x),\label{lambda-limit-2}\\
&\lim_{\lambda \to +\infty} u(x + \lambda \nu_2) = u_2 (P_2 x),\label{lambda-limit-3}
\end{align}
%\end{equation}
for $x \in \R^3$. Here, $u$ is an equivariant solution produced by Theorem \ref{theorem1}, $P_0$, $P_1$, and $P_2$ are projections on the hyperplanes orthogonal to $\nu_0$, $\nu_1$, $\nu_2$ respectively, and $u_0$, $u_1$, $u_2$ are solutions to \eqref{system}, with $u_i : \R^i \to \R^3$, for $i = 0,1,2$, equivariant with respect to the subgroup $G_{\nu_i}$ of $\mathcal{T}$ leaving $\nu_i$ fixed, and connecting respectively $a_1$ to itself, $a_1$ to $a_2$, and $a_1$, $a_2$, and $a_3$ (in the sense of Theorem \ref{theorem1}).

The work of Gui and Schatzman \cite{gui-schatzman} is restricted to the tetrahedral group and their solution \emph{by construction} satisfies the list above. We note that one of the differences of Theorem \ref{theorem1} and the corresponding theorem in \cite{alikakos-fusco-arma} is that our solutions are not hierarchical by construction. It is intractable to build in the hierarchy for general groups or higher dimensions. \emph{A posteriori}, however, one expects to prove the complete hierarchical structure for certain groups. Theorem \ref{theorem1} already contains \eqref{lambda-limit-1}, with $u_0 \equiv a_1$. 

Below we state a general theorem establishing \eqref{lambda-limit-2} in the context of Theorem \ref{theorem1}. We formulate the theorem in an independent way. Consider the problem
\begin{equation}\label{problem}
\begin{cases}
\Delta u - W_u (u) = 0, &\text{in } \Omega \subset \R^n,\smallskip\\
u=u_0, &\text{on } \partial \Omega,
\end{cases}
\end{equation}
for $u: \Omega \to \R^m$, where $\Omega$ is a smooth open domain and $W$ is a $C^2$ potential. We have the following hypotheses on $W$ and $\Omega$.

\begin{h}\label{h1}
The potential $W$ is symmetric, with $W(u_1, u_2, \dots, u_m) = W(-u_1, u_2, \dots, u_m)$.
\end{h}

\begin{h}\label{h2}
There exists a nondegenerate minimum $a_+$ of the potential $W$ (cf.\ Hypothesis \ref{hypothesis1} in Section \ref{section3}) such that $a_1 > 0$ and $0=W(a_1) \leq W(u)$, for $u \in \R^m$, and such that
\[ W(a_+) < W(u), \]
for $|u - a_+| \leq q_0$, for some $q_0 >0$, with $u \neq a_+$.
\end{h}

\begin{h}\label{h3}
There holds
\[ u(Tx) = T u(x), \]
where $T$ is the reflection with respect to the first coordinate in either $x$ or $u$, that is, $Tx = (-x_1, x_2, \dots, x_n)$ and $Tu = (-u_1, u_2, \dots, u_m)$. Moreover, $u_0 (Tx) = T u_0(x)$.
\end{h}

\begin{h}\label{h4}
The domain $\Omega$ is globally Lipschitz and convex symmetric, in the sense that
\[ (x_1, x_2, \dots, x_n) \in \Omega \quad \text{implies} \quad (tx_1, x_2, \dots, x_n) \in \Omega, \]
for all $t$, with $|t| \leq 1$.
\end{h}

Finally, we have the following hypotheses on the solution and the connecting orbit.

\begin{h}\label{h5}
A solution $u: \Omega \to \R^m$ of problem \eqref{problem} is a \emph{global minimizer} if
\[
J(u; \Omega') = \min J(v; \Omega'), \text{ such that } v=u \text{ on } \partial \Omega',
\]
over the class of bounded and smooth domains $\Omega' \subset \Omega$ and over all equivariant maps $v(Tx) = Tv(x)$ in $W^{1,2} (\Omega; \R^m)$. (Cf.\ \eqref{global-minimizer} in Section \ref{section3}.)
\end{h}

\begin{h}\label{h6}
We assume that there is a unique orbit $U: \R \to \R^m$ connecting the minima $a_{\pm} = (\pm a_1, a_2, \dots, a_m)$ such that
\[ U'' - W_u(U) = 0, \text{ with } U(\pm \infty) = a_{\pm}, \]
which is also \emph{hyperbolic} in the class of symmetric variations $v(Tx) = Tv(x)$, for $v: \R \to \R^m$, that is, zero is not in the spectrum of the linearized operator
\[ Lv := v'' - W_{uu} (U)v, \text{ for symmetric } v \in W^{1,2} (\R; \R^m). \]
\end{h}

Under the above hypotheses, we have the following theorem.

\begin{theorem}[\cite{alikakos-fusco-hierarchy}]\label{theorem-new}
Assume that hypotheses {\rm (H\ref{h1})--(H\ref{h6})} hold. Then, there holds the estimate
\[ |u(x) - U(x_1)| \leq K \e^{-k \dist(x,\partial \Omega)}, \]
for $x \in \Omega$ and $K, k$ positive constants.
\end{theorem}

Finally, we state a theorem under no hypotheses of uniqueness or hyperbolicity of the connecting orbit. Our notation is that of Theorem \ref{theorem1} in Section \ref{section3} and $u$ is a solution as in that theorem.

\begin{theorem}[\cite{alikakos-bates}]
Let $\Pi_1$ be a wall of $D$ and assume that $a_2$ is the reflection of $a_1$ with respect to $\Pi_1$. Moreover, assume that the set of orbits connecting $a_1$ to $a_2$ is nonempty. Then, there exists a $\nu_1 \in \interior \Pi_1$, with $|\nu_1|=1$, and a sequence $\{ \lambda _k \} \to +\infty$ such that
\[ u(x + \lambda_k \nu_1) \to U(P_1 x), \]
where $U$ is a connection between $a_1$ and $a_2$, and $P_1$ is the orthogonal projection to ${\Pi_{1}}^{\!\bot}$.
\end{theorem}

\begin{remark}
One would expect that a stronger version of the above theorem holds: Given any $\nu_1 \in \interior (\Pi_1)$ and any sequence $\lambda_k \to +\infty$ there exists a subsequence $\{ {\lambda_{k}}' \}$ of $\{ \lambda_{k} \}$ such that $u(x + {\lambda_{k}}' \nu_1) \to U(P_1 x)$.
\end{remark}

\nocite{*}
\bibliographystyle{plain}

\end{document}